\documentclass[10pt]{amsart}

\usepackage{amsmath,amsfonts,amsthm,mathrsfs,amssymb,verbatim,cite,cases,fouridx}
\usepackage[left=3.5cm,right=3.5cm,top=3.5cm,bottom=3.5cm]{geometry}
\usepackage[usenames]{color}
\usepackage[dvipsnames]{xcolor}
\usepackage{hyperref}
\usepackage{bbm}
\usepackage{mathtools}
\usepackage{tikz,graphicx}
\usetikzlibrary{hobby,math}
\usepackage{todonotes}
\usepackage{upgreek}
\usepackage{accents}

\newtheorem{theorem}{Theorem}
\newtheorem{lemma}[theorem]{Lemma}

\newtheorem{conjecture}[theorem]{Conjecture}
\newtheorem*{problem}{Problem}
\theoremstyle{remark}
\newtheorem*{remark}{Remark} %[theorem]{Remark}

\renewcommand{\leq}{\leqslant}
\renewcommand{\geq}{\geqslant}
\newcommand{\ceil}[1]{\lceil#1\rceil}
\newcommand{\floor}[1]{\lfloor#1\rfloor}
\newcommand{\qbin}[2]{\genfrac{[}{]}{0pt}{}{#1}{#2}}
\newcommand{\abs}[1]{\vert#1\vert}
\newcommand{\dash}{\,---\,}
\newcommand{\ba}{\boldsymbol{a}}
\newcommand{\bb}{\boldsymbol{b}}

\begin{document}

\title{$q$-rious unimodality}

\author{S.~Ole Warnaar}

\address{School of Mathematics and Physics,
The University of Queensland, Brisbane, Australia}
%\email{o.warnaar@maths.uq.edu.au}
%\urladdr{https://people.smp.uq.edu.au/OleWarnaar/}

\author{Wadim Zudilin}
\address{Institute for Mathematics, Astrophysics and Particle Physics,
Radboud University, Nijmegen, The Netherlands}
%\email{w.zudilin@math.ru.nl}
%\urladdr{http://www.math.ru.nl/~wzudilin/}

\dedicatory{To the memory of Dick Askey, mentor and friend}


\thanks{Work supported by the Australian Research Council (SOW) and the Max-Planck-Institute for Mathematics (WZ)}

\begin{abstract}
We generalise our still-wide-open \emph{$q$-rious positivity conjecture} from 2011 to a \emph{$q$-rious unimodality conjecture}.
\end{abstract}

\maketitle

Anyone who has ever had a mathematics conversation with the late great Dick Askey, or anyone who has ever attended one of his talks (or a talk where he was in the audience) would know that he loved to remind us of simple unsolved problems that he thought we should be thinking more about.
`Simple' of course refers to `simple to state' but almost always hard to solve.
Although we would never dream of comparing ourselves to Dick Askey, it nonetheless seems appropriate in a tribute paper to follow in his footsteps and remind combinatorialists, special functions aficionados and number theorists of the following simple open problem, \cite[Conjecture 1]{WZ11}.

We say that a nonzero polynomial $c_0+c_1q+\dotsb+c_kq^k\in\mathbb{Z}[q]$ is \emph{positive} if all $c_i\geq 0$.  

\begin{conjecture}[$q$-rious positivity]\label{qrious}
Let $\ba=(a_1,\dotsc,a_r)$ and $\bb=(b_1,\dotsc,b_s)$ be tuples of positive integers satisfying
\begin{equation}\label{Landau}
\sum_{i=1}^r\floor{a_ix}-\sum_{j=1}^s\floor{b_jx}\geq 0
\quad\text{for all}\; x\geq 0.
\end{equation}
Then the polynomial
\begin{equation}\label{Dab}
D(\ba,\bb;q)=\frac{[a_1]!\dotsm[a_r]!}{[b_1]!\dotsm[b_s]!}
\end{equation}
is positive.
\end{conjecture}

In the above, for $n$ a nonnegative integer, $[n]!$ is the $q$-analogue of $n!$, defined as
\[
[n]!=[n]_q!:=[1][2]\dotsm [n],
\]
where an empty product should be taken as $1$ and $[n]=[n]_q:=(1-q^n)/(1-q)$ is a $q$-number.
The difference $s-r$ is known as the \emph{height} of the $q$-factorial ratio \eqref{Dab} and the inequality \eqref{Landau} as \emph{Landau's criterion} \cite{Landau85}.
This criterion is a necessary and sufficient condition for the integrality of $D_n(\ba,\bb;1)$ for all positive integers $n$, where
\[
D_n(\ba,\bb;q):=D(\ba\,n,\bb\,n;q)=\frac{[a_1n]!\dotsm[a_rn]!}{[b_1n]!\dotsm[b_sn]!}.
\]
The irreducible factors over $\mathbb{Q}$ of $[n]$ are given by cyclotomic polynomials, and by a simple analysis of the latter it follows that $D(\ba,\bb;q)\in\mathbb{Z}[q]$ and, more generally, $D_n(\ba,\bb;q)\in\mathbb{Z}[q]$, see \cite{WZ11} for details.

Although Conjecture~\ref{qrious} was posed almost 15 years ago, and no counterexamples have yet been found, there are very few irreducible parametric families for which the conjecture has been proved.
Before discussing these cases and stating a generalisation to Conjecture~\ref{qrious}, we make a few preliminary comments.

A pair of tuples of positive integers $\ba=(a_1,\dotsc,a_r)$ and $\bb=(b_1,\dotsc,b_s)$ is said to be \emph{balanced} if $\abs{\ba}-\abs{\bb}=0$, where $\abs{\ba}:=a_1+\dotsb+a_r$ and $\abs{\bb}:=b_1+\dotsb+b_s$.
Now let us temporarily refer to \eqref{Landau} as Landau's criterion I, or \emph{Landau I} for short, and define \emph{Landau II} as the condition
\begin{equation}\label{Landau2}
\sum_{i=1}^r\floor{a_ix}-\sum_{j=1}^s\floor{b_jx}\geq 0
\quad\text{for all}\; x\in\mathbb{R}.
\end{equation}

\begin{lemma}
If $(\ba,\bb)$ satisfies Landau II then Landau I holds and $(\ba,\bb)$ is balanced.
\end{lemma}

\begin{proof}
It is clear that Landau II implies Landau I. 
To show balancing, note that
\begin{equation}\label{floor}
\sum_{i=1}^r \floor{a_ix}-\sum_{j=1}^s \floor{b_jx}=(\abs{\ba}-\abs{\bb})x-
\sum_{i=1}^r \{a_ix\}+\sum_{j=1}^s \{b_jx\},
\end{equation}
where $\{x\}$ is the fractional part of $x$.
Since the sum of differences between fractional parts is bounded and
$(\abs{\ba}-\abs{\bb})x$ is unbounded from below unless $\abs{\ba}=\abs{\bb}$, balancing follows.
\end{proof}

\begin{lemma}
If $(\ba,\bb)$ is balanced and satisfies Landau I then Landau II holds.
\end{lemma}

\begin{proof}
If balancing holds then \eqref{floor} simplifies to
\[
\sum_{i=1}^r \floor{a_ix}-\sum_{j=1}^s \floor{b_jx}=
-\sum_{i=1}^r \{a_ix\}+\sum_{j=1}^s \{b_jx\}.
\]
Since this is $1$-periodic in $x$, Landau I implies Landau II.
\end{proof}

The above two lemmas imply that Landau I and Landau II are equivalent for balanced $(\ba,\bb)$. 
But if $(\ba,\bb)$ satisfies Landau I and is unbalanced, then it follows from \eqref{floor} that $\abs{\ba}-\abs{\bb}\geq 1$.
Defining
\[
\bb'=(b'_1,\dotsc,b'_{s'})=(b_1,\dotsc,b_s,\underbrace{1,1,\dotsc,1}_{\text{$s'-s$ times}})
\quad\text{with}\; s'=s+\abs{\ba}-\abs{\bb}
\] 
yields a new pair $(\ba,\bb')$ that is balanced.
Moreover, for $x\in[0,1)$,
\[
\sum_{i=1}^r \floor{a_ix}-\sum_{j=1}^{s'} \floor{b_j'x}=
\sum_{i=1}^r \floor{a_ix}-\sum_{j=1}^s \floor{b_jx}-(s'-s)\floor{x}=
\sum_{i=1}^r \floor{a_ix}-\sum_{j=1}^s \floor{b_jx}\geq 0.
\]
By the $1$-periodicity for balanced pairs, the above inequality thus holds for all $x\in\mathbb{R}$, so that Landau II holds for $(\ba,\bb')$.
The upshot of the above discussion is that it suffices to consider balanced pairs $(\ba,\bb)$ in the remainder, in which case Landau I and Landau II are of course indistinguishable.

In the remainder we say that a pair $(\ba,\bb)$ satisfying Landau's criterion is \emph{coprime} if
\[
\gcd(\ba,\bb):=\gcd(a_1,\dots,a_r,b_1,\dots,b_s)=1.
\]
Clearly, by replacing $x\mapsto x/d$, if $(\ba,\bb)$ satisfies \eqref{Landau} then so does $(\ba/d,\bb/d)$, where $d=\gcd(\ba,\bb)$.

\begin{theorem}[\!\!{\cite{WZ11}}]
\label{th-WZ}
Conjecture~\ref{qrious} holds for all balanced coprime pairs $(\ba,\bb)$ of height one.
\end{theorem}

The proof is essentially case-by-case. 
Bober \cite{Bober09} gave a complete classification of all irreducible coprime pairs $(\ba,\bb)$ of height one.
His list consists of $52$ sporadic cases (including Chebyshev's $\ba=(1,30)$, $\bb=(6,10,15)$) for which it is easy to do a computer-assisted check of positivity, and the three two-parameter families:
\begin{subequations}
\begin{align}
(\ba,\bb)&=\big((m+n),(m,n)\big), \label{qbt} \\
(\ba,\bb)&=\big((2m,n),(m,m-n,2n)\big), \label{super} \\
(\ba,\bb)&=\big((2m,2n),(m,n,m+n)\big) \label{Bonkers2},
\end{align}
\end{subequations}
where in each case $m,n$ are relatively prime positive integers and $m>n$ in for the second family.
For these three families it was shown in \cite{WZ11} that the polynomials $D(\ba,\bb;q)$, given by
\begin{subequations}
\begin{align}
&\qquad\frac{[m+n]!}{[m]![n]!}, \label{q-bin} \\
B(m,n;q) &:=\frac{[2m]![n]!}{[m]![m-n]![2n]!}, \label{Bonkers} \\
C(m,n;q) &:=\frac{[2m]![2n]!}{[m]![m+n]![n]!}, \label{Catalan}
\end{align}
\end{subequations}
are positive for all pairs of nonnegative integers $m,n$ and $m\geq n$ in the case of \eqref{Bonkers}.
The first family is classical and corresponds to the well-known $q$-binomial coefficients $\qbin{m+n}{m}$.
Positivity is trivial and follows from the recursion
\[
\qbin{m+n}{m}=\qbin{m+n-1}{m-1}+q^m \qbin{m+n-1}{m}
\]
(and $\qbin{m}{0}=\qbin{m}{m}=1$), or from the well-known fact that $\qbin{m+n}{m}$ is the generating function for integer partitions contained in an $m\times n$ rectangle~\cite{Andrews76}.
For the other two families of polynomials no combinatorial interpretation is known, making their positivity somewhat more mysterious, although in both cases positivity is not hard to establish using recurrences.
The case $B(m,n;q)$ is most akin to the $q$-binomial coefficients, admitting the three-term recursion
\begin{equation}\label{Bonkers-rec}
B(m,n;q)=q^{2m-2n}B(m-1,n-1;q)+(1+q^m+q^{m-n}+q^{2m+n-1})B(m-1,n;q),
\end{equation}
which combined with $B(m,0;q)=\qbin{2m}m$ and $B(m,m;q)=1$ implies positivity.
The third family of polynomials corresponds to the $q$-analogue of the super Catalan numbers.
For these no manifestly positive three-term recursion is presently known, see \cite{WZ11}.

It is not hard to find further irreducible balanced pairs $(\ba,\bb)$ satisfying Landau's criterion.
There is the partial classification by Soundararajan of balanced coprime pairs of height two \cite{Soundararajan20,Soundararajan22}, and also root systems that are not simply laced are a good source of examples.
However, proving positivity of non-sporadic cases appears to be extremely difficult.
This difficulty extends to Dick Askey's favourite pair~\cite{Askey86}
\[
(\ba,\bb)=\big((2m,2n,3n,3m+3n),(m,n,n,m+n,m+2n,2m+3n)\big),
\]
which arises in the Macdonald--Morris constant term identity for the root system $\mathrm{G}_2$ \cite{Habsieger86,Macdonald82,Morris82,Zeilberger87}, and its $\mathrm{F}_4$ counterpart \cite{GG92,Macdonald82,Morris82}
\begin{align*}
(\ba,\bb)&=\big((2m,2n,3m,3n,4n,2m+4n,4m+2n,2m+6n,4m+4n,6m+6n), \\
&\quad\qquad (m,m,n,n,n,m+n,m+2n,2m+n,m+3n,2m+3n,\\
&\quad\qquad\qquad 3m+3n,3m+4n,3m+5n,5m+6n)\big).
\end{align*}
In these examples, the height is exactly the rank of the root system.
Of course, since the classification of irreducible balanced coprime pairs satisfying \eqref{Landau} is still completely open (and for arbitrary height is almost surely intractable), a case-by-case approach will never settle the conjecture.
What is really needed is a proof that the integrality of $D_n(\ba,\bb)$ for all positive integers $n$ implies $D(\ba,\bb;q)$ is a positive polynomial.  

\medskip

Following the old adage ``if you cannot prove it, generalise it'', we in the following propose a new and equally `simple' conjecture that implies the $q$-rious positivity.

Assume that $(\ba,\bb)$ satisfies Landau's criterion.
Then 
\[
D(\ba,\bb;q)=\sum_{i=0}^k c_i q^i,
\]
where $c_i\in\mathbb{Z}$, $c_0=1$, $c_i=c_{k-i}$ for all $0\leq i\leq k$ and
$k:=\sum_i \binom{a_i}{2}-\sum_j \binom{b_j}{2}$.
Here the symmetry of the coefficients is an immediate consequence of
\[
[n]_{1/q}!=q^{-\binom{n}{2}} [n]_q!.
\]
The problem is thus to show that $c_i$ is nonnegative for all $0\leq i\leq \floor{k/2}$.

A finite sequence $(c_0,c_1,\dotsc,c_k)$ is \emph{symmetric and unimodal} if $c_i=c_{k-i}$ for all $i$ and 
\[
c_0\leq c_1\leq\cdots\leq c_{\floor{k/2}}.
\]
Accordingly, the polynomial $P(q)=\sum_{i=0}^k c_i q^i$ is symmetric and unimodal if its coefficient sequence is symmetric and unimodal.
Clearly, if $P(q)$ and $Q(q)$ are two such polynomials, then so is their product.

\begin{conjecture}[$q$-rious unimodality]\label{qnimodal}
Let $(\ba,\bb)$ satisfy Landau's criterion \eqref{Landau}.
Then the polynomial $(1+\nobreak q) D(\ba,\bb;q)$ is unimodal.
\end{conjecture}

\begin{lemma}
Conjecture~\ref{qnimodal} implies Conjecture~\ref{qrious}.
\end{lemma}

\begin{proof}
Let $P(q):=D(\ba,\bb;q)$ and $Q(q):=(1+q)P(q)$, and assume that $P(q)$ has even order:
$P(q)=\sum_{i=0}^{2k} c_iq^i$, where $c_i=c_{2k-i}$ and $c_0=1$.
Then
\[
Q(q)=
\sum_{i=0}^k (c_{i-1}+c_i)q^i\big(1+q^{2k-2i+1}\big),
\]
where $c_{-1}:=0$.
Since $Q(q)$ is unimodal
\[
1=c_0\leq c_2\leq\cdots\leq c_{2\floor{k/2}}
\quad\text{and}\quad 
0\leq c_1\leq c_3\leq\cdots\leq c_{2\ceil{k/2}-1}.
\]
This establishes the positivity of $P(q)$.
The odd-order case proceeds in almost identical fashion and is left to the reader.
\end{proof}

\begin{remark}
A polynomial $P(q)=c_0+c_1 q+\dots+c_k q^k$ is said to be \emph{parity unimodal} if the sequences $(c_0,c_2,\dots,c_{2\floor{k/2}})$ and $(c_1,c_3,\dots,c_{2\ceil{k/2}-1})$ are both unimodal, see \cite{Stucky21} (or \cite[Conjecture 7]{IO15}). 
Assuming Conjecture~\ref{qnimodal} it follows that $D(\ba,\bb;q)$ is parity unimodal. 
Parity unimodality of a polynomial $P(q)$, however, is necessary but not sufficient for the unimodality of $Q(q)=(1+q)P(q)$.
For example, if $P(q)=\sum_{i=0}^{2k+1} c_i q^i$ such that $c_i=c_{2k-i+1}$, then $Q(q)$ is unimodal if and only if
\[
c_0\leq c_2\leq\cdots\leq c_{2\floor{k/2}}\leq c_{2\ceil{k/2}-1}\geq\cdots\geq c_3\geq c_1\geq 0;
\]
this clearly is stronger than the parity unimodality of $P(q)$.
If $P(q)=\sum_{i=0}^{2k} c_i q^i$ such that $c_i=c_{2k-i}$, then $Q(q)$ is unimodal if and only if
\[
c_0\leq c_2\leq\dots\leq c_{2\floor{k/2}}
\quad\text{and}\quad
c_{2\ceil{k/2}-1}\geq\cdots\geq c_3\geq c_1\geq 0.
\]
For $P(q)=D(\ba,\bb;q)$ the condition $c_1\geq 0$ trivially holds.
Therefore, if this polynomial has even degree, then parity unimodality implies 
the unimodality of $Q(q)$.
\end{remark}

As a companion to Theorem~\ref{th-WZ} we have the following, which requires no more than routine computer assisted verification.

\begin{lemma}
Conjecture~\ref{qnimodal} holds for all $52$ sporadic irreducible coprime pairs $(\ba,\bb)$ in Bober's list as well as for the two-parameter family \eqref{qbt}.
\end{lemma}

We remark that for most sporadic pairs unimodality holds without the factor $1+q$. 
The exceptions are the 16 pairs labelled 1, 2, 3, 5, 6, 10, 11, 12, 23, 24, 32, 39, 40, 44, 50, 52 in Bober's classification.
Of course, unimodality (without the factor $1+q$) also holds for the $q$-binomial coefficients \eqref{q-bin}, a result first proved by Sylvester \cite{Sylvester78}.
None of the many subsequent proofs of unimodality \dash including the algebraisation of O'Hara's combinatorial proof \cite{Ohara90} by Zeilberger, see e.g., \cite{Bressoud89,Macdonald89,Zeilberger89} \dash can be considered elementary.
This, combined with the lack of a combinatorial or representation theoretic interpretation of \eqref{Bonkers} and \eqref{Catalan}, perhaps explains why so far we have been unable to show that the conjecture holds for the pairs \eqref{super} and \eqref{Bonkers2}.
In the case of the $q$-super Catalan polynomials \eqref{Catalan} it would seem that Conjecture~\ref{qnimodal} holds without the factor $1+q$, whereas for \eqref{Bonkers} this factor appears to be necessary only for
$(m,n)\in\{(4,1),(5,1),(7,1),(8,1),(8,3)\}$ and $(m,n)\in\{(2k+1,2k-1):~k\geq 1\}$, as well as for their symmetric counterparts.
For example,
\begin{align*}
B(4,1;q)=\frac{[1]![8]!}{[2]![3]![4]!}
&=1+q+3q^2+4q^3+7q^4+8q^5+12q^6+12q^7+15q^8+14q^9 \\
&\quad
+15q^{10}+12q^{11}+12q^{12}+8q^{13}+7q^{14}+4q^{15}+3q^{16}+q^{17}+q^{18}.
\end{align*}
Extensive computer checks also strongly suggest that $D_n(\ba,\bb;q)$ is unimodal for all integers $n\geq 2$.

To conclude this note, we discuss the evidence we have in support of Conjecture~\ref{qnimodal}.
For each of Bober's sporadic pairs, there is just a single polynomial for which unimodality has to be checked, making verification a straightforward, finite task.
For the corresponding 52 non-coprime families $D_n(\ba,\bb;q)$ for $n\geq 2$, the sporadic height-two examples found by Soundararajan, as well as for all of the available two-parametric and three-parametric families of solutions to \eqref{Landau}, we have typically checked all those polynomials within a family of order up to $10,000$.
This list of parametric families extends well beyond Bober's height-one classification, the height-two family of \cite{Soundararajan20}
\begin{equation}\label{Sound}
\frac{[6m]![n]!}{[2m]![3m]![m-5n]![6n]!} \quad\text{for $m\geq 5n$},
\end{equation}
(which includes Chebyshev's factorial ratio $D_n\big((1,30),(6,10,15);1)$),
and the examples arising from root systems, such as the $\mathrm{G}_2$ and $\mathrm{F}_4$ cases.
Additional families may be sourced from the many summation formulas for basic hypergeometric series, see e.g.,~\cite{GR04}.
For example, according to the $q$-analogue of Dixon's summation \cite{Andrews75}
\[
\frac{[\ell+m+n]!\,[2\ell]!\,[2m]!\,[2n]!}{[\ell]!\,[m]!\,[n]!\,[\ell+m]!\,[m+n]!\,[n+\ell]!}
=\sum_k(-1)^kq^{k(3k-1)/2}\qbin{2\ell}{\ell+k}\qbin{2m}{m+k}\qbin{2n}{n+k}.
\]
Since the right-hand side asserts the polynomiality of the $q$-factorial ratio on the left (although not its positivity), this ratio is a good candidate for unimodality, an observation that is supported by computer checks.
Furthermore, by extrapolating some of the findings from \cite{Bober09,Soundararajan20,Soundararajan22}, one is naturally led to consider the following two infinite collections of two-parameter families:
\[
B_{\lambda}(m,n;q):=\frac{[\abs{\lambda}m+m]![n]!}
{[\lambda m]![m-\abs{\lambda}n]![\abs{\lambda}n+n]!} \quad\text{for}\; m\geq\abs{\lambda}n
\]
and
\[
C_\lambda(m,n;q):=\frac{[\abs{\lambda}m+m]![\abs{\lambda}n+n]!}{[\lambda m]!\,[m+\abs{\lambda}n]![n]!}.
\]
Here $\lambda$ is an integer partition $\lambda=(\lambda_1\geq\lambda_2\geq\dotsb\geq\lambda_r\geq 1)$ of size $\abs{\lambda}=\lambda_1+\lambda_2+\dotsb+\lambda_r$ and $[\lambda m]!:=\prod_{i=1}^r [\lambda_i m]!$.
In order for \eqref{Landau} to be satisfied, additional restrictions on $\lambda$ need to be imposed and, up to size $11$, there are $14$ admissible partitions:
\begin{gather*}
(1), \; (1,1), \; (2,1), \; (2,1,1), \; (3,2), \; (3,2,1), \; (4,2,1), \; (4,3,1), \\ (5,3,1), \; (5,2,2), \; (4,3,2), \; (5,3,2), \; (6,4,1), \; (4,4,3).
\end{gather*}
$B_{(1)}(m,n;q)$ and $C_{(1)}(m,n;q)$ are precisely \eqref{Bonkers} and \eqref{Catalan}, while $B_{(3,2)}(m,n;q)$ is \eqref{Sound}.
For the other two partitions of length two we get
\begin{align*}
B_{(1,1)}(m,n;q)&=\frac{[3m]![n]!}{[m]![m]![m-2n]![3n]!}, \\
B_{(2,1)}(m,n;q)&=\frac{[4m]![n]!}{[m]![2m]![m-3n]![4n]!},
\end{align*}
generalising
\[
B_{(1,1)}(2n,n;q)=D_n\big((1,6),(2,2,3);q\big)=B(3n,n;q)
\]
and
\[
B_{(2,1)}(3n,n)=D_n\big((1,12),(3,4,6);q\big)=B(6n,n;q).
\]
Conjecture~\ref{qnimodal} is supported by all these additional examples.

Given the difficulty of establishing unimodality (or even positivity), we conclude by posing an easier problem that hopefully is accessible through asymptotic methods.

\begin{problem}
Let $(\ba,\bb)$ be any balanced coprime pair satisfying Landau's criterion \eqref{Landau}.
\begin{enumerate}
\item[1.] Show that there exists a positive integer $N$ such that $D_n(\ba,\bb;q)$ is positive/unimodal for all $n\geq N$.
\item[2.] Find an explicit upper bound for $N$ in terms of $(\ba,\bb)$. 
\end{enumerate}
\end{problem}

\subsection*{Acknowledgement}
We thank Christian Krattenthaler and Vic Reiner for alerting us to the papers \cite{IO15,Stucky21} in which polynomials displaying parity unimodality are considered.
We further thank Matthew Bolan for pointing out the $q=1$ instance of the three-term recursion for $B(m,n;q)$.
From this it was easy to find \eqref{Bonkers-rec}, which is quite different from the recursion used in \cite{WZ11} to establish positivity.

\end{document}